\documentclass[12pt]{article}
\usepackage{amsmath,amssymb,amsbsy,amsfonts,amsthm,latexsym,
             amsopn,amstext,amsxtra,euscript,amscd}

\begin{document}
\bibliographystyle{plain}

\newfont{\teneufm}{eufm10}
\newfont{\seveneufm}{eufm7}
\newfont{\fiveeufm}{eufm5}
%
%
\newfam\eufmfam
               \textfont\eufmfam=\teneufm \scriptfont\eufmfam=\seveneufm
               \scriptscriptfont\eufmfam=\fiveeufm
%
%
\def\frak#1{{\fam\eufmfam\relax#1}}
%


\def\bbbr{{\rm I\!R}} 
\def\bbbm{{\rm I\!M}}
\def\bbbn{{\rm I\!N}} 
\def\bbbf{{\rm I\!F}}
\def\bbbh{{\rm I\!H}}
\def\bbbk{{\rm I\!K}}
\def\bbbp{{\rm I\!P}}
\def\bbbone{{\mathchoice {\rm 1\mskip-4mu l} {\rm 1\mskip-4mu l}
{\rm 1\mskip-4.5mu l} {\rm 1\mskip-5mu l}}}
\def\bbbc{{\mathchoice {\setbox0=\hbox{$\displaystyle\rm C$}\hbox{\hbox
to0pt{\kern0.4\wd0\vrule height0.9\ht0\hss}\box0}}
{\setbox0=\hbox{$\textstyle\rm C$}\hbox{\hbox
to0pt{\kern0.4\wd0\vrule height0.9\ht0\hss}\box0}}
{\setbox0=\hbox{$\scriptstyle\rm C$}\hbox{\hbox
to0pt{\kern0.4\wd0\vrule height0.9\ht0\hss}\box0}}
{\setbox0=\hbox{$\scriptscriptstyle\rm C$}\hbox{\hbox
to0pt{\kern0.4\wd0\vrule height0.9\ht0\hss}\box0}}}}
\def\bbbq{{\mathchoice {\setbox0=\hbox{$\displaystyle\rm
Q$}\hbox{\raise
0.15\ht0\hbox to0pt{\kern0.4\wd0\vrule height0.8\ht0\hss}\box0}}
{\setbox0=\hbox{$\textstyle\rm Q$}\hbox{\raise
0.15\ht0\hbox to0pt{\kern0.4\wd0\vrule height0.8\ht0\hss}\box0}}
{\setbox0=\hbox{$\scriptstyle\rm Q$}\hbox{\raise
0.15\ht0\hbox to0pt{\kern0.4\wd0\vrule height0.7\ht0\hss}\box0}}
{\setbox0=\hbox{$\scriptscriptstyle\rm Q$}\hbox{\raise
0.15\ht0\hbox to0pt{\kern0.4\wd0\vrule height0.7\ht0\hss}\box0}}}}
\def\bbbt{{\mathchoice {\setbox0=\hbox{$\displaystyle\rm
T$}\hbox{\hbox to0pt{\kern0.3\wd0\vrule height0.9\ht0\hss}\box0}}
{\setbox0=\hbox{$\textstyle\rm T$}\hbox{\hbox
to0pt{\kern0.3\wd0\vrule height0.9\ht0\hss}\box0}}
{\setbox0=\hbox{$\scriptstyle\rm T$}\hbox{\hbox
to0pt{\kern0.3\wd0\vrule height0.9\ht0\hss}\box0}}
{\setbox0=\hbox{$\scriptscriptstyle\rm T$}\hbox{\hbox
to0pt{\kern0.3\wd0\vrule height0.9\ht0\hss}\box0}}}}
\def\bbbs{{\mathchoice
{\setbox0=\hbox{$\displaystyle     \rm S$}\hbox{\raise0.5\ht0\hbox
to0pt{\kern0.35\wd0\vrule height0.45\ht0\hss}\hbox
to0pt{\kern0.55\wd0\vrule height0.5\ht0\hss}\box0}}
{\setbox0=\hbox{$\textstyle        \rm S$}\hbox{\raise0.5\ht0\hbox
to0pt{\kern0.35\wd0\vrule height0.45\ht0\hss}\hbox
to0pt{\kern0.55\wd0\vrule height0.5\ht0\hss}\box0}}
{\setbox0=\hbox{$\scriptstyle      \rm S$}\hbox{\raise0.5\ht0\hbox
to0pt{\kern0.35\wd0\vrule height0.45\ht0\hss}\raise0.05\ht0\hbox
to0pt{\kern0.5\wd0\vrule height0.45\ht0\hss}\box0}}
{\setbox0=\hbox{$\scriptscriptstyle\rm S$}\hbox{\raise0.5\ht0\hbox
to0pt{\kern0.4\wd0\vrule height0.45\ht0\hss}\raise0.05\ht0\hbox
to0pt{\kern0.55\wd0\vrule height0.45\ht0\hss}\box0}}}}
\def\bbbz{{\mathchoice {\hbox{$\sf\textstyle Z\kern-0.4em Z$}}
{\hbox{$\sf\textstyle Z\kern-0.4em Z$}}
{\hbox{$\sf\scriptstyle Z\kern-0.3em Z$}}
{\hbox{$\sf\scriptscriptstyle Z\kern-0.2em Z$}}}}
\def\ts{\thinspace}

\newtheorem{theorem}{Theorem}
\newtheorem{lemma}[theorem]{Lemma}
\newtheorem{claim}[theorem]{Claim}
\newtheorem{cor}[theorem]{Corollary}
\newtheorem{prop}[theorem]{Proposition}
\newtheorem{definition}[theorem]{Definition}
\newtheorem{remark}[theorem]{Remark}
\newtheorem{question}[theorem]{Open Question}

\def\qed{\ifmmode
\squareforqed\else{\unskip\nobreak\hfil
\penalty50\hskip1em\null\nobreak\hfil\squareforqed
\parfillskip=0pt\finalhyphendemerits=0\endgraf}\fi}

\def\squareforqed{\hbox{\rlap{$\sqcap$}$\sqcup$}}

\def\cA{{\mathcal A}}
\def\cB{{\mathcal B}}
\def\cC{{\mathcal C}}
\def\cD{{\mathcal D}}
\def\cE{{\mathcal E}}
\def\cF{{\mathcal F}}
\def\cG{{\mathcal G}}
\def\cH{{\mathcal H}}
\def\cI{{\mathcal I}}
\def\cJ{{\mathcal J}}
\def\cK{{\mathcal K}}
\def\cL{{\mathcal L}}
\def\cM{{\mathcal M}}
\def\cN{{\mathcal N}}
\def\cO{{\mathcal O}}
\def\cP{{\mathcal P}}
\def\cQ{{\mathcal Q}}
\def\cR{{\mathcal R}}
\def\cS{{\mathcal S}}
\def\cT{{\mathcal T}}
\def\cU{{\mathcal U}}
\def\cV{{\mathcal V}}
\def\cW{{\mathcal W}}
\def\cX{{\mathcal X}}
\def\cY{{\mathcal Y}}
\def\cZ{{\mathcal Z}}
\newcommand{\rmod}[1]{\: \mbox{mod}\: #1}

\def\tcN{\cN^\mathbf{c}}

\def\Tr{{\mathrm{Tr}}}

\def\mand{\qquad \mbox{and} \qquad}
\renewcommand{\vec}[1]{\mathbf{#1}}

\def\eqref#1{(\ref{#1})}


\newcommand{\ignore}[1]{}

\hyphenation{re-pub-lished}

\parskip 1.5 mm
\def\lln{{\mathrm Lnln}}
\def\Res{\mathrm{Res}\,}

\def\F{{\bbbf}}
\def\Fp{\F_p}
\def\fp{\Fp^*}
\def\Fq{\F_q}
\def\ff{\F_2}
\def\ffn{\F_{2^n}}

\def\K{{\bbbk}}
\def \Z{{\bbbz}}
\def \N{{\bbbn}}
\def\Q{{\bbbq}}
\def \R{{\bbbr}}

\def\Zm{\Z_m}
\def \Um{{\mathcal U}_m}

\def \Bf{\frak B}

\def\Km{\cK_\mu}

\def\va {{\mathbf a}}
\def \vb {{\mathbf b}}
\def \vc {{\mathbf c}}
\def\vx{{\mathbf x}}
\def \vr {{\mathbf r}}
\def \vv {{\mathbf v}}
\def\vu{{\mathbf u}}
\def \vw{{\mathbf w}}
\def \vz {{\mathbfz}}

\def\\{\cr}
\def\({\left(}
\def\){\right)}
\def\fl#1{\left\lfloor#1\right\rfloor}
\def\rf#1{\left\lceil#1\right\rceil}

\def\flq#1{{\left\lfloor#1\right\rfloor}_q}
\def\flp#1{{\left\lfloor#1\right\rfloor}_p}
\def\flm#1{{\left\lfloor#1\right\rfloor}_m}

\def\Al{{\sl Alice}}
\def\Bob{{\sl Bob}}

\def\Or{{\mathcal O}}

\def\inv#1{\mbox{\rm{inv}}\,#1}
\def\invM#1{\mbox{\rm{inv}}_M\,#1}
\def\invp#1{\mbox{\rm{inv}}_p\,#1}

\def\Ln#1{\mbox{\rm{Ln}}\,#1}

\def \nd {\,|\hspace{-1.2mm}/\,}

\def\ord{\mu}

\def\E{\mathbf{E}}

\def\Cl{{\mathrm {Cl}}}

\def\epp{\mbox{\bf{e}}_{p-1}}
\def\ep{\mbox{\bf{e}}_p}
\def\eq{\mbox{\bf{e}}_q}

\def\bm{\bf{m}}

\newcommand{\floor}[1]{\lfloor {#1} \rfloor}

\newcommand{\comm}[1]{\marginpar{%
\vskip-\baselineskip 
\raggedright\footnotesize
\itshape\hrule\smallskip#1\par\smallskip\hrule}}

\def\rem{{\mathrm{\,rem\,}}}
\def\dist {{\mathrm{\,dist\,}}}
\def\etal{{\it et al.}}
\def\ie{{\it i.e. }}
\def\veps{{\varepsilon}}
\def\eps{{\eta}}

\def\ind#1{{\mathrm {ind}}\,#1}
                \def \MSB{{\mathrm{MSB}}}
\newcommand{\abs}[1]{\left| #1 \right|}

\newcommand{\set}[1]{\left\{#1\right\}}
\def\tf {\widetilde{f}}
\def\tg {\widetilde{g}}
\def\th {\widetilde{h}}


\title{Pseudorandom Bits From
Points on Elliptic Curves}

\author{
          {\sc{Reza Rezaeian Farashahi and Igor E.~Shparlinski}} \\
          {Department of Computing}\\
          {Macquarie University} \\
          {Sydney, NSW 2109, Australia} \\
          {\tt \{reza,igor\}@ics.mq.edu.au}
          }

\maketitle

\begin{abstract}
Let $\E$ be an elliptic curve over a finite field $\F_{q}$ of
$q$ elements, with $\gcd(q,6)=1$, given by an affine Weierstra\ss\ equation.
We also use $x(P)$ to denote the $x$-component of a point $P =
(x(P),y(P))\in \E$.
We  estimate character sums of the form
$$
\sum_{n=1}^N  \chi\(x(nP)x(nQ)\) \quad \text{and}\quad
\sum_{n_1, \ldots, n_k=1}^N \psi\(\sum_{j=1}^k c_j  x\(\(\prod_{i =1}^j n_i\) R\)\)
$$
on average over all $\F_q$ rational points  $P$, $Q$ and $R$
on $\E$, where $\chi$ is a quadratic character, $\psi$ is a nontrivial
additive character in $\F_q$ and $(c_1, \ldots, c_k)\in \F_q^k$ is a
non-zero vector.  These bounds confirm several
recent conjectures of D.~Jao, D.~Jetchev and R.~Venkatesan,
related to extracting random bits from various sequences of points
on elliptic curves.
\end{abstract}


\section{Introduction}

\subsection{Motivation}

Many standard pseudorandom number generators based
on finite fields and residue rings have proved to be
insecure, see~\cite{BGGS1,BGGS2,BGGS3,Boy1,Boy2,ContShp,
FHKLS,vzGShp,GoGuIb,JS,Kraw,Lag}. Partially motivated
by this and partially because this is of intrinsic interest
for elliptic curve cryptography, several constructions of
pseudorandom generators from elliptic curves have been proposed,
see~\cite{Shp} for a survey of such constructions
and results.

Several new pseudorandom generators from elliptic curves 
 have recently been suggested
by D.~Jao, D.~Jetchev and R.~Ven\-katesan~\cite{JJV}.
Giving a rigorous analysis of these constructions is
the primal goal of this paper. We also show how one
of the most powerful number theoretic techniques,
exponential and character sums, can be used to address
these and similar questions, which can be of independent
interest.

Finally, we note that although elliptic curves provide a very promising
source of cryptographically secure bits, as the recent result of~\cite{GuIb}
shows, they also  have to be used with great care.

\subsection{Results}

We fix a finite field $\F_{q}$ of
$q$ elements and  an elliptic curve $\E$ over $\F_{q}$
given by an affine Weierstra\ss\ equation
\begin{equation}
\label{eq:Weier}
\E:\quad   Y^2 = X^3  + aX + b
\end{equation}
with some $a,b \in \F_{q}$, see~\cite{Silv}.

We recall that the set of all points on $\E$ forms an Abelian group,
with the point at infinity $\cO$ as the neutral
element,  and
we use $\oplus $ to denote the group operation.
As usual we write every point $P \ne \cO$ on $\E$ as $P = (x(P), y(P))$.
For $P=\cO$ we formally write $P=(0, \infty)$.

Let $\E(\F_{q})$   denote the set of $\F_q$-rational
points on  $\E$.

For a positive integer $N$,  points
$P, Q, R \in \E(\F_{q})$,
and a non-zero vector
$\vec{c} = (c_1, \ldots, c_k)\in \F_q^k$,
define character sums of the form
\begin{eqnarray*}
S(P,Q;N) &= &\sum_{n=1}^N  \chi\(x(nP)x(nQ)\), \\
T_k(\vec{c},R;N) &= & \sum_{n_1, \ldots, n_k=1}^N \psi\(\sum_{j=1}^k
c_j x\(\(\prod_{i =1}^j n_i\) R\)\),
\end{eqnarray*}
where $\chi$ is a quadratic character
(we also put $\chi(0) = 0$) and
$\psi$ is a nontrivial additive
characters in~$\F_q$.

D.~Jao, D.~Jetchev and R.~Venkatesan~\cite[Conjecture~4.1]{JJV}
have conjectured
that there exists a positive constant $\delta>0$ such that
for any  $N \ge (\log q)^2$ and any points $P \ne Q$ the bound
$$
S(P,Q;N) = O(N^{1-\delta})
$$
holds.
Towards this conjecture, it has been shown in~\cite[Section~4.2]{JJV}
that for any point $Q\in \E(\F_{q})$,
$$
\sum_{P\in \E(\F_{q})} S(P,Q;N) = O(qN^{1/2}).
$$
This however does not imply that the sums $S(P,Q;N)$ are typically,
or even sometimes, small. Furthermore, the proof given in~\cite{JJV}
seems to hold only if the cardinality
$\# \E(\F_{q})$ is not divisible by any prime $\ell \le N$.
Here we use a different argument and estimate
the sum
$$
U(N) = \sum_{P,Q\in \E(\F_{q})} |S(P,Q;N)|^2,
$$
which immediately implies that the sums $S(P,Q;N)$
are small for almost all pairs of points $P,Q\in \E(\F_{q})$.

We also estimate the average value of the sums $T_k(\vec{c},R;N)$ over points
of subgroups $\cH \subseteq \E(\F_{q})$
of order $t$ which  is not divisible by any prime $\ell \le N$.
Namely for a subgroup $\cH$ of the group of points $\E(\F_{q})$,
we estimate the sum
$$
V_k(\vec{c},\cH; N) = \sum_{R\in \cH} |T_k(\vec{c},R;N)|^2
$$
which similarly  implies
that the sums $T_k(\vec{c},R;N)$
are small for almost all  points $R\in \cH$.
Note that subgroups of cryptographic interest are usually chosen
to be of a prime order, so the   coprimality condition
$\gcd(N!, \# \cH)=1$ is always satisfied.

In turn, in the case of prime $q=p$, we derive
from our bound on $V_k(\vec{c},\cH; N)$  that for almost all   points $R\in
\E(\F_{q})$,
strings of $\ell$ least significant bits of each components
of the  $k$-dimensinal points
\begin{eqnarray}
\label{eq:vect}
\left\{x\(\(\prod_{i =1}^j n_i\) R\)\right\}_{j=1}^k,\qquad n_1, \ldots, n_k \in \{1, \ldots, N\},
\end{eqnarray}
are uniformly distributed (provided that $\#\cH$ is large enough).
We note that instead of strings of most significant bits (as
suggested in~\cite{JJV})
we use least significant bits. This is because for some primes $p$
(for those which are very close to a power of $2$)
most significant bits of random residues modulo $p$ are biased,
while
least significant bits are always uniformly distributed.
A step towards such a result is made in~\cite[Proposition~5.1]{JJV}
but it contains some parameters which are not explicitly
estimated in~\cite{JJV} (and as we have just mentioned it
cannot work for most significant bits anyway).

Throughout the paper, the implied constants in symbols `$O$' and
`$\ll$'  are absolute  (we recall that $U\ll V$
and $U = O(V)$ are both equivalent to the inequality $|U|\le cV$ with some
constant $c> 0$).

\bigskip

\noindent{\bf Acknowledgements.}
This work was supported in part by ARC Grant DP0881473, Australia,
(for R.R.F. and I.S.)
and by NRF Grant~CRP2-2007-03, Singapore, (for I.S).

\section{Preparations}

\subsection{Backgrounds on division polynomials}

For an integer $n\ge 0$, let $\psi_n(X,Y)$ be the $n$th {\it division
polynomial\/} of $\E$ over $\F_{q}$ given by~\eqref{eq:Weier}, we
refer to~\cite{Silv} for a
background on  division polynomials.
Let
\begin{equation}\label{eq: def fn}
f_n=X\psi^2_n-\psi_{n-1}\psi_{n+1} \mand  g_n=\psi_n^2, \qquad n =1,2, \ldots.
\end{equation}
In particular, $f_n$ and $g_n$ are polynomials in $\F_{q}[X]$
of degrees
\begin{equation}
\label{eq:Deg}
\begin{array}{ll}
\deg f_n = n^2 \mand
\deg g_n \le n^2-1,\\
\end{array}
\end{equation}
such that
\begin{equation}
\label{eq:xfg}
x(nP) = \frac{ f_n(x(P))}{g_n(x(P))}.
\end{equation}
Further, one can write
\begin{equation}\label{eq:gn}
g_n(X) =
\left\{
\begin{array}{ll}
h^2_n(X), & \text{ if $n$ is odd},\\
(X^3+aX+b)h^2_n(X), & \text{ if $n$ is even},\\
\end{array}
\right.
\end{equation}
for some polynomials $h_n(X)$ in $\F_{q}[X]$, $n =1,2, \ldots$.

It is well known, and  also follows from~\eqref{eq:xfg},
that the roots of the polynomial $g_n$, for $n \ge
2$, are the $x$-coordinates of $n$-torsion points of $\E$, that is,
for all points $P$ in $\E(\overline{\F}_q)$ with $P\ne \cO$, we have
$$P=(x,y) \in \E[n] \Longleftrightarrow g_n(x)=0,$$
where, as usual,
$$
\E[n]=\set{P~:~P\in \E(\overline{\F}_q),\ nP=\cO}.
$$
and $\overline{\F}_q$ denotes the algebraic closure of $\F_q$.

We note that, if $\gcd(n,q)=1$, then
$$\E[n] \cong \Z/n\Z \times \Z/n\Z.$$
Moreover, if $\F_q$ is of characteristic $p$,
then $\E[p]$ is isomorphic to $\Z/p\Z$ or $\set{\cO}$.
We recall that an
elliptic curve $\E$ is called ordinary if $\E[p]\cong \Z/p\Z$. It is
called supersingular if $\E[p] \cong \set{\cO}$. 
Furthermore, if $p$ divides $n$, write $n=p^rn_*$ with $\gcd(p,n_*)=1$. Then
\begin{equation*}
\E[n]=\E[n_*] \oplus  \E[p^r],
\end{equation*}
where $\E[p^r] \cong \Z/p^r\Z$ if $\E$ 
is ordinary and $\E[p^r] \cong \set{\cO}$ if $\E$ is supersingular. In particular, $\#\E[n]=nn_*$ if $\E$ is ordinary and $\#\E[n]=n_*^2$ if $\E$ is supersingular.

Denote the set of $n$-division points of a point $Q$ in $\E$ by
$\E[n,Q]$, that is,
$$\E[n,Q]=\set{P~:~P\in \E(\overline{\F}_q), nP=Q}.
$$
Clearly, $nP=Q$ if and only if $\E[n,Q]=P \oplus \E[n]$.

The following result shows that the roots of $f_n$ are the $x$-coordinates of $n$-division points of a point $P_0$ on $\E$ with $x(P_0)=0$.

\begin{lemma}
\label{lem:E-fn}
Let $\E$ be an elliptic curve over $\F_q$ given by the 
equation~\eqref{eq:Weier}. Let $P_0=(0,c)\in \E(\overline{\F}_q) $, 
where $c$ is a square root of $b$. Then, for 
all $x\in \overline{\F}_q$, we have $f_n(x)=0$ 
if and only if there exist a point 
$P \in \E[n,P_0]$ with $x(P)=x$.
\end{lemma}
 
\begin{proof}
Let $x\in \overline{\F}_q$. Then, there exists an element 
$y\in \overline{\F}_q$ such that the point 
$P=(x,y)$ is a point on $\E$. If $f_n(x)=0$, then $g_n(x)\ne 0$. 
Moreover, from~\eqref{eq:xfg}, we have $x(nP)=0$. 
So, $nP=P_0$ or $nP=-P_0$. Thus, $nP=P_0$ or $n(-P)=P_0$, 
that is, either $P=(x,y)$ or $-P=(x,-y)$ is a point of $\E[n,P_0]$.

If $P=(x,y)\in \E[n,P_0]$, then $nP=P_0$. So, $x(nP)=x(P_0)=0$. 
 Next, from~\eqref{eq:xfg}, we have $f_n(x)=0$.
\end{proof}

\begin{lemma}\label{lem:fn}
For all positive integers $n=p^rn_*$ with $\gcd(n_*,p)=1$, we have
$$
f_n(X) =
\left\{
\begin{array}{ll}
\widetilde{f}_n(X)^{p^r}, & \text{ if $\E$ is ordinary},\vspace{2mm}\\
\widetilde{f}_n(X)^{p^{2r}}, & \text{ if $\E$ is supersingular},\\
\end{array}
\right.
$$
for some polynomial $\widetilde{f}_n$ in $\F_q[X]$ with $\deg
\widetilde{f}_n=\#\E[n]$.
\end{lemma}
\begin{proof}
We note that, for $n=p^rn_*$, $f_n$ is a polynomial of $X^{p^r}$ if
$\E$~is ordinary~(for example, see~\cite[Lemma~2]{ChHa}). Moreover,
$f_n$ is a polynomial of $X^{p^{2r}}$ if $\E$~is supersingular~(for
example, see~\cite{ChHa}). Recalling~\eqref{eq:Deg},
we see that  if $\E$ is ordinary, one can write
$f_n=\widetilde{f}_n(X)^{p^r}$, for some polynomial $\widetilde{f}_n$
in $\F_q[X]$ of degree $p^r{n_*}^2$. If  $\E$ is supersingular, then
$f_n=\widetilde{f}_n(X)^{p^{2r}}$, for some polynomial
$\widetilde{f}_n$ in $\F_q[X]$ of degree ${n_*}^2$. In other words,
$\deg \widetilde{f}_n=\#\E[n]$.
\end{proof}

\begin{lemma}\label{lem:fn b=0}
If $b\ne 0$, then for all positive integers $n$ the polynomial
$\widetilde{f}_n$, defined by Lemma~\ref{lem:fn}, is square-free.
\end{lemma}

\begin{proof}
>From Lemma~\ref{lem:E-fn}, we see that the roots of $f_n$ are the
$x$-coordinates of points of $\E[n,P_0]$. Then, from
Lemma~\ref{lem:fn}, we also see that the roots of $\widetilde{f}_n$
are the
$x$-coordinates of points of $\E[n,P_0]$. We note that, for $P\in
\E[n,P_0]$, the point $-P$ is in $\E[n,P_0]$ if and only if
$P_0=-P_0$, that is, $-P \in \E[n,P_0]$ if and only if $b=0$. So, if
$b\ne 0$, all points of $\E[n,P_0]$ have distinct $x$-coordinates. We
note that, $\#\E[n,P_0]=\#\E[n]$. Hence, the polynomial
$\widetilde{f}_n$ has $\#\E[n]$ distinct roots. From Lemma~\ref{lem:fn},
$\deg \widetilde{f}_n=\#\E[n]$. Therefore, if $b \ne0$, the
polynomial $\widetilde{f_n}$ is square-free.
\end{proof}

We now define the rational functions
\begin{equation}
\label{eq:Phi Psi}
\begin{split}
\Phi_{m,n}(X) &= \frac{f_m(X)f_n(X)}{g_m(X)g_n(X)}, \\
\Psi_{m,n}(X) &= \frac{(X^3 + aX + b) f_m(X)f_n(X)}{g_m(X)g_n(X)}.
\end{split}
\end{equation}

We need the following property of  $\Phi_{m,n}$ and
$\Psi_{m,n}$, which can be of independent interest.

\begin{lemma}\label{lem:fnfm}
If $\E$ is an ordinary elliptic curve with $b\ne 0$, then for all
distinct positive integers $m$ and $n$,
  neither $\Phi_{m,n}$ nor $\Psi_{m,n}$ is   a square of a
rational function in $\overline{\F}_q(X)$.
\end{lemma}

\begin{proof}
{From}~\eqref{eq:Deg} and~\eqref{eq:gn}, we see that the difference of
$\deg f_n$ and $\deg g_n$ is odd. So, the difference between the
degrees of the numerator and denominator of $\Psi_{m,n}$
is odd. So, it cannot  be a square of another rational
function.

For $\Phi_{m,n}$, first, we assume that $m+n$ is
even. From~\eqref{eq:gn}, we see that $g_mg_n$ is a square.
Let $m={p^r}m_*$ and $n=p^sn_*$ with $\gcd(m_*n_*,p)=1$.
By Lemmas~\ref{lem:fn} and~\ref{lem:fn b=0}, we write
$f_m=\widetilde{f}_m^{p^r}$ and $f_n=\widetilde{f}_n^{p^s}$, where
the polynomials $\widetilde{f}_m$, $\widetilde{f}_n$ are square-free.
Moreover, $\deg \widetilde{f}_m=p^r{m_*}^2$ and $\deg
\widetilde{f}_n=p^s{n_*}^2$. So, for distinct $m$, $n$, $\deg
\widetilde{f}_m\ne \deg \widetilde{f}_n$. Thus,
$\widetilde{f}_m\widetilde{f}_n$ can not be a square of a polynomial
in $\overline{\F}_q[X]$. The same is true for the product of $f_m$
and $f_n$. Hence, $\Phi_{m,n}$ can not be a square of a rational
function.

Now, we assume that $m+n$ is odd. From~\eqref{eq:gn}, we have
$g_mg_n=(X^3+aX+b)h^2_mh^2_n$.
We recall that the roots of $X^3+aX+b$ are corresponded to the
$x$-coordinates of points of $\E[2]$. Also, the roots of $f_m$ are
corresponded to the $x$-coordinates of points of $\E[m,P_0]$. Clearly
the sets $\E[2]$ and $\E[m,P_0]$ have no common point if $b\ne0$.
Therefore, $X^3+aX+b$ has no common root with $f_m$ and similarly
with $f_n$ where $b\ne 0$. So, again $\Phi_{m,n}$ can not be a square
of a rational function.
\end{proof}

\subsection{Exponential Sums Along Elliptic Curves}

We recall
the following bound of character sums with
  a nontrivial additive character $\psi$ of $\F_q$,
which is given  in~\cite{LanShp}.

\begin{lemma}
\label{lem:ExpSum}
Fix integers $1 \le d_1 <  \ldots < d_s \le D$
and fix   $c_1, \dots, c_s\in \F_q$ with $c_s\ne0$.
Let $\E$ be an ordinary elliptic curve defined over $\F_q$.
Then  the following bound holds:
$$
  \sum_{\substack{Q \in \cH\\ Q \ne \cO}}
\psi\(\sum_{i = 1}^s c_i x\(d_iQ\)\)   = O\( sD^2 q^{1/2}\),
$$
where  $\cH$ is an arbitrary subgroup of  $\E(\F_q)$
of order $t = \# \cH$ such that
$$
\gcd (t, d_1\cdots d_s) = 1.
$$
\end{lemma}

\section{Main Results}

\subsection{Sums $U(N)$}

\begin{theorem}\label{thm:UN}
For a prime power $q$ with $\gcd(q,6)=1$ and an ordinary elliptic
curve $\E$ given by~\eqref{eq:Weier}
with $b \ne 0$, we have
$$
U(N) \ll N^6 q + N q^2
$$
for every positive integer $N$.
\end{theorem}

\begin{proof} Expanding the square and changing the order of summation,
we obtain
\begin{eqnarray*}
U(N) &= &\sum_{m,n=1}^N \sum_{P,Q\in \E(\F_{q})}
\chi\(x(mP)x(nP)x(mQ)x(nQ)\) \\
   &= &  \sum_{m,n=1}^N \left|\sum_{P \in \E(\F_{q})}
\chi\(x(mP)x(nP)\)\right|^2 .
\end{eqnarray*}

For $n=m$, we estimate the inner sum over $P$ trivially as $O(q)$.
Thus the total contribution to $U(N)$ from such terms is
\begin{equation}
\label{eq: m eq n}
U^{(=)}(N) = O(N q^2).
\end{equation}

If $n\ne m$, as in~\cite[Section~4.2]{JJV}  we note that any  $u \in \F_q$
appears as $u = x(P)$ for some point $P \in \E(\F_{q})$ exactly
$1 + \chi(u^3 + au + b)$ times,
where $a$ and $b$ are as in~\eqref{eq:Weier}.
Therefore, using~\eqref{eq:xfg}, we derive
$$
\sum_{P \in \E(\F_{q})}  \chi\(x(mP)x(nP)\)
=  \sum_{u \in \F_q} \chi\(\Phi_{m,n}(u)\)
+  \sum_{u \in \F_q} \chi\(\Psi_{m,n}(u)\),
$$
where the polynomials $\Phi_{m,n}(X)$ and
$\Psi_{m,n}(X)$ are given by~\eqref{eq:Phi Psi}.

Now, by Lemma~\ref{lem:fnfm}, we see that the Weil
bound applies to both sums, see~\cite[Theorems~11.23]{IwKow},
and together with~\eqref{eq:Deg} leads to the estimate
$$
\sum_{P \in \E(\F_{q})}  \chi\(x(mP)x(nP)\) = O\(N^2 q^{1/2}\)
$$
for $n \ne m$.
Thus the total contribution to $U(N)$ from such terms is
\begin{equation}
\label{eq: m ne n}
U^{(\ne)}(N) = O\(N^2\(N^2 q^{1/2}\)^2\)= O\(N^6 q\).
\end{equation}

Combining~\eqref{eq: m eq n} and~\eqref{eq: m ne n},
we finish the proof.
\end{proof}

Clearly, Theorem~\ref{thm:UN} improves the trivial bound
$U(N) \ll  N^2 q^2$ for $N \le q^{1/4 -\delta}$
with any fixed $\delta> 0$. This is well within
the range of interest in~\cite{JJV} which starts with $N$
of order $(\log q)^2$.
Furthermore, if $N \le q^{1/5}$ then the bounds takes
the form
$U(N) \ll Nq^2$, thus confirming that for almost all
$P,Q \in  \E(\F_{q})$ the sums $S(P,Q;N)$ have square
root cancellations (see comments after~\cite[Conjecture~4.1]{JJV}).

\subsection{Sums $V_k(\vec{c},\cH; N)$}

We note that an appropriate version of  the results
of this section holds
for any $q$ (in fact even without the condition
$\gcd(q,6) =1$). However, to make our argument more transparent,
  we assume that
$q= p$ is prime. It is exactly the case which is
needed for our prime goal, which is studying the bit patterns
of the vectors~\eqref{eq:vect}.

\begin{theorem}\label{thm:VN}
For   a prime $p$,  an ordinary elliptic  curve $\E$
and a  subgroup $\cH$  of  $\E(\F_p)$
of order $t$, uniformly
over all non-zero vectors $\vec{c} \in \F_p^k$, we have
$$
V_k(\vec{c},\cH; N) \ll kN^{4k} p^{1/2} +  k N^{2k-1} t
$$
for all  positive integers $k$ and $N$ with $\gcd(N!,t) = 1$.
\end{theorem}

\begin{proof}  Squaring out, expanding and changing the order of
summation, we obtain
\begin{equation}
\begin{split}
\label{eq: mult sum R}
V_k(\vec{c},\cH; N) & =  \sum_{m_1, \ldots, n_k=1}^N
  \sum_{n_1, \ldots, n_k=1}^N  \\
  &  \quad \sum_{R\in \cH} \psi\(\sum_{j=1}^k c_j x\(\(\prod_{i =1}^j m_i\) R\)
  - \sum_{j=1}^k c_j x\(\(\prod_{i =1}^j n_i\) R\)\).
\end{split}
\end{equation}

For $O(kN^{2k-1})$ choices of  $m_1, \ldots, m_k$ and $n_1, \ldots, n_k$
with at least
one value equal to 1 we estimate the inner sum trivially as $t$.
So the total contribution from such terms is
\begin{equation}
\label{eq:V1}
V_1 \ll kN^{2k-1}t.
\end{equation}

We say that the sequence of integers
$ m_1, \ldots, m_k, n_1, \ldots, n_k\ge 2 $
is product distinct with respect to
$\vec{c}$ the vectors
$$
\(m_1, m_1m_2, \ldots, m_1m_2\ldots m_k\)
\mand\(n_1, n_1n_2, \ldots, n_1n_2\ldots n_k\)
$$
distinct at all positions $j$ for which $c_j \in \F_p^*$.

We see from Lemma~\ref{lem:ExpSum} that if
$ m_1, \ldots, m_k, n_1, \ldots, n_k\ge 2 $
is product distinct with respect to
$\vec{c}$ then  the inner
sum over $R$ in~\eqref{eq: mult sum R} is 
$O\(kN^{2k} p^{1/2}\)$.
Otherwise we estimate this sum trivially as $O(t)$.

The total
contribution from these terms   is
\begin{equation}
\label{eq:V2 prelim}
V_2 \ll k N^{4k} p^{1/2} + M t.
\end{equation}
where $M$ is the number of
sequence of integers
$ N  \ge m_1, \ldots, m_k, n_1, \ldots, n_k\ge 2 $
which are not product distinct with respect to
$\vec{c}$.

To estimate $M$,
we assume that $c_{j_0} \ne 0$.  If all values
of $m_1, \ldots, m_k$
and all values of $n_1, \ldots , n_k$, but
$n_{j_0}$ are fixed,
then $n_{j_0}$ must satisfy the
equation
$$
m_1\ldots m_{j_0} = n_1\ldots n_{j_0}
$$
and thus can take at most one possible value.
Since $j_0$
takes $k$ distinct values,
the total contribution we get $M \le  k N^{2k-1}$
(the vector $\vec{c} = (1,0, \ldots, 0)$
shows
that this bound cannot be improved).
Substituting this bound in~\eqref{eq:V2 prelim} we obtain
\begin{equation}
\label{eq:V2}
V_2 \ll k N^{4k} p^{1/2} +  k N^{2k-1} t.
\end{equation}

Combining~\eqref{eq:V1} and \eqref{eq:V2},
we conclude the proof.
\end{proof}

\subsection{Applications}

We now address the question of~\cite{JJV} on the distribution
of bits of the vectors~\eqref{eq:vect}.

Let now $q=p$ be prime. We assume that  $\F_p$
is represented by the elements of the set $\{0, 1, \ldots, p-1\}$.

For a point $R \in \E(\F_p)$,
positive integers $k$, $\ell$, $N$ and  $k$ bit strings $\sigma_1,
\ldots, \sigma_k$
of length $\ell$ each, we  use
$A_{k,\ell}(R, N; \sigma_1,  \ldots, \sigma_k)$
to denote
the number of times the least significant  bits of
the binary expansions of the
components of the vectors~\eqref{eq:vect}
are $\sigma_1,  \ldots, \sigma_k$, respectively.
It is natural to compare
 $A_{k,\ell}(R, N; \sigma_1,  \ldots, \sigma_k)$
with  $2^{-k \ell}N^k$. Thus, for a subgroup $\cH \subseteq \E(\F_p)$,
we consider the average deviation $\Delta_{k,\ell}(\cH, N)$
of
$A_{k,\ell}(R, N; \sigma_1,  \ldots, \sigma_k)$ from
its expected value:
$$
\Delta_{k,\ell}(\cH, N) =
  \sum_{R\in \cH}
  \max_{\sigma_1,  \ldots, \sigma_k} \left| A_{k,\ell}(R, N; \sigma_1,
\ldots, \sigma_k) -  2^{-k \ell} N^k \right|,
$$
where the maximum is taken over all $2^{k\ell}$ choices
of $k$ bit strings $\sigma_1,  \ldots, \sigma_k$ of length $\ell$.

\begin{theorem}\label{thm:DN}
There is an absolute constant $C > 0$
such that for a prime $p> k$,  an ordinary curve $\E$
and a  subgroup $\cH$  of  $\E(\F_p)$
of order $t$, uniformly
over all non-zero vectors $\vec{c} \in \F_p^k$, we have
$$
\Delta_{k,\ell}(\cH, N)\le \(N^{2k} p^{1/4}t^{1/2} + N^{k-1/2} t\) (C \log p)^k
$$
for all positive integers $k$, $\ell$ and $N$ with $\gcd(N!,t) = 1$.
\end{theorem}

\begin{proof} Clearly the binary expansion of $x \in \F_p$
ends with an $\ell$-bit string $\sigma$ if and only if
$x = 2^\ell y + \bar{\sigma}$,  where $\bar{\sigma}$
is the integer represented by $\sigma$
and the   integer $y$ is
such that
$0 \le y < (p - \bar{\sigma})/2^\ell$.
Alternatively, denoting by $\lambda \in \F_p$ the reciprocal
of $2^\ell$, we obtain
$$
\lambda(x-\bar{\sigma}) = y .
$$

We now define
$$
L_j = \rf{(p - \bar{\sigma_j})/2^\ell} - 1,
  \qquad j =1, \ldots, k.
$$
We also recall  the identity
$$
\frac{1}{ p}\sum_{c \in \F_p} \psi(cv) =
\left\{\begin{array}{ll}
1,&\quad\text{if $v=0$,}\\
0,&\quad\text{if $v\in \F_p^*$.}
\end{array}
\right.
$$

Therefore, for any fixed  nontrivial additive character $\psi$ of $\F_p$,
we have
\begin{eqnarray*}
\lefteqn{A_{k,\ell}(R, N; \sigma_1,  \ldots, \sigma_k)}\\
 & & \qquad    \quad =
  \sum_{n_1, \ldots, n_k=1}^N
  \sum_{y_1=0}^{L_1} \ldots  \sum_{y_k=0}^{L_k}\\
  & &  \qquad \qquad \qquad \qquad
  \prod_{j =1}^k \frac{1}{p}\sum_{c_j\in \F_p}
  \psi\( c_j\(\lambda x\(\(\prod_{i =1}^j n_i\) R\) - \lambda
\bar{\sigma_j} -y_j\)\)\\
  & &   \qquad  \quad = \sum_{n_1, \ldots, n_k=1}^N  \sum_{y_1=0}^{L_1} \ldots  \sum_{y_k=0}^{L_k}\\
  & &   \qquad \qquad \qquad \frac{1}{p^k} \sum_{\vec{c} \in \F_p^k} \prod_{j =1}^k \psi\( \lambda c_j x\(\(\prod_{i =1}^j n_i\) R\)\)\psi(- \lambda c_j
\bar{\sigma_j})\psi( -c_jy_j)\\
  & &   \qquad \quad = \frac{1}{p^k} \sum_{\vec{c} \in \F_p^k}
T_k(\lambda \vec{c},R;N)
  \psi\(- \lambda\sum_{j =1}^k   c_j \bar{\sigma_j}\)
  \prod_{j =1}^k \sum_{y_j=0}^{L_j}  \psi\(- c_jy_j\),
  \end{eqnarray*}
where the outer summation is taken over all vectors
$\vec{c} = (c_1, \ldots, c_k) \in  \F_p^k$.
Separating the term
\begin{eqnarray*}
\frac{(L_1+1)\ldots (L_k+1) N^k}{p^k}
& = &2^{-k\ell}N^k + O\(k  2^{-(k-1)\ell}N^k p^{-1}\)\\
& = & 2^{-k\ell}N^k + O\(N^k p^{-1}\),
\end{eqnarray*}
corresponding to the zero-vector $\vec{c} = \vec{0}$, we obtain
\begin{eqnarray*}
\lefteqn{\left|A_{k,\ell}(R, N; \sigma_1,  \ldots, \sigma_k)
-  2^{-k\ell}N^k \right| }\\
& &  \qquad \qquad \qquad \ll N^k p^{-1} + \frac{1}{p^k}
\sum_{\substack{\vec{c} \in \F_p^k\\\vec{c} \ne \vec{0}}}
\left|
T_k(\lambda \vec{c},R;N)\right|
\prod_{j =1}^k \left| \sum_{y_j=0}^{L_j}  \psi\(- c_jy_j\)\right|.
\end{eqnarray*}
Furthermore, using that
$$
  \sum_{y=0}^{L}  \psi\(- cy\) \ll  \frac{p}{1+\min\{c, p-c\}},
$$
which holds for $c \in \F_p$ and a positive integer $L$,
see~\cite[Bound~(8.6)]{IwKow}, we derive
\begin{eqnarray*}
\lefteqn{\left|A_{k,\ell}(R, N; \sigma_1,  \ldots, \sigma_k)
-  2^{-k\ell}N^k \right|}\\
  & &  \qquad \quad  \ll N^k p^{-1} +
  \sum_{\substack{\vec{c} \in \F_p^k\\\vec{c} \ne \vec{0}}}
\left|
T_k(\lambda \vec{c},R;N)\right|
  \prod_{j =1}^k \frac{1}{1+\min\{c_j, p-c_j\}}.
  \end{eqnarray*}
Since the right hand side of the last expression
does not depend on $\sigma_1,  \ldots, \sigma_k$, we see that
$$\Delta_{k,\ell}(\cH, N)  \ll  N^k t p^{-1} +
  \sum_{\substack{\vec{c} \in \F_p^k\\\vec{c} \ne \vec{0}}}  \prod_{j
=1}^k \frac{1}{1+\min\{c_j, p-c_j\}} \sum_{R\in \cH}
\left|
T_k(\lambda \vec{c},R;N)\right| .
$$
Finally, using  the Cauchy inequality and then applying Theorem~\ref{thm:VN},
we obtain
\begin{eqnarray*}
\lefteqn{\Delta_{k,\ell}(\cH, N)}\\
  & & \quad  \ll N^k t p^{-1} +
  \sum_{\substack{\vec{c} \in \F_p^k\\\vec{c} \ne \vec{0}}} \sqrt{t V_k(\vec{c},\cH; N)}
 \prod_{j=1}^k \frac{1}{1+\min\{c_j, p-c_j\}} \\
  & & \quad  \ll N^k t p^{-1} +\(N^{2k} p^{1/4}t^{1/2} + k^{1/2}N^{k-1/2} t\)
\prod_{j =1}^k \sum_{c_j \in \F_p} \frac{1}{1+\min\{c_j, p-c_j\}} .
\end{eqnarray*}
We choose $C_0$ such that
$$
\sum_{c \in \F_p} \frac{1}{1+\min\{c, p-c\}} \le C_0 \log p.
$$
Taking $C > C_0$ sufficiently large
(to accommodate in $C^k$ all other constants and also the factor $k^{1/2}$)
we obtain
$$
\Delta_{k,\ell}(\cH, N)\le \(N^{2k} p^{1/4}t^{1/2} +N^k t p^{-1}+ N^{k-1/2} t\) (C \log p)^k.
$$
Furthermore, the condition $\gcd(N!,t)=1$
implies that $N < t = O(p)$
thus $ N^k t p^{-1} \ll N^{k-1/2} t$.
Hence the term $ N^k t p^{-1}$
can be omitted from the above bound,
which concludes the proof.
\end{proof}

We recall that in~\cite{JJV}, it has been suggested to
use the values $N = (\log p)^{O(1)}$.
Since cardinalities of elliptic curves of cryptographic interest are
either prime or contain a very small smooth part (that is, a part
composed out of small primes), it is natural to assume that  the
order $t$ of the largest subgroup   $\cH$  of  $\E(\F_p)$
with $\gcd(N!,t) = 1$ satisfies $t \sim p^{1 + o(1)}$. In fact,
assuming only that $t\ge p^{1/2 + \delta}$ for some
fixed $\delta > 0$, we see that Theorem~\ref{thm:DN}
is nontrivial provided $k \ell = o(\log N)$ and
asserts that for almost all points $R \in \cH$, strings of $\ell$
least significant bits
of the vectors~\eqref{eq:vect} are uniformly distributed.
That is,
for all $2^{k\ell}$ choices
of $k$ bit strings $\sigma_1,  \ldots, \sigma_k$ of length
$\ell$ for almost all points $R \in \cH$, the counting function
$A_{k,\ell}(R, N; \sigma_1,\ldots, \sigma_k)$ is
close to its expected value $2^{-k \ell} N^k$.

\end{document}